\documentclass[reqno,12pt]{amsart}
\usepackage{amssymb}
\usepackage{amsmath}
\usepackage{amsthm}
\usepackage{pgf}
\usepackage{graphicx}

\newtheorem{theorem}{Theorem}[section]
\newtheorem{lemma}[theorem]{Lemma}
\newtheorem{proposition}[theorem]{Proposition}

\newtheorem{remark}[theorem]{Remark}

\renewcommand{\det}{{\rm det}\,}

\renewcommand{\O}{\Omega}

\newcommand{\md}{{\rm d}}
 
\renewcommand{\psi}{\mathfrak{Z}}

\newcommand{\Dcal}{\mathcal{D}}

\newcommand{\R}{{\mathbb R}} 
\newcommand{\N}{{\mathbb N}}

\newcommand{\be}{\begin{eqnarray}}
\newcommand{\ee}{\end{eqnarray}}

\newcommand{\Rz}{{\mathbb R}}
\newcommand{\haz}{\widehat}

\usepackage{cancel}
\usepackage{xifthen}
 
\begin{document}
 
\title[ Existence results for incompressible magnetoelasticity]{
 Existence results for  incompressible
magnetoelasticity}
\author[M. Kru\v z\'\i k]{Martin Kru\v z\'\i k}
\author[U. Stefanelli]{Ulisse Stefanelli}
\author[J. Zeman]{Jan Zeman}

\address[M. Kru\v z\'\i k]{Institute of Information Theory and Automation, Academy of Sciences of the Czech Republic,
Pod vod\'{a}renskou v\v{e}\v{z}\'{\i}~4, CZ-182 08 Praha 8,
Czech Republic and
Faculty of Civil Engineering, Czech Technical University, Th\'{a}kurova 7, CZ--166~29 Praha~6, Czech Republic.}
\email{kruzik@utia.cas.cz}

\address[U. Stefanelli]{Faculty of Mathematics, University of Vienna, Oskar-Morgenstern-Platz 1, A-1090 Vienna, Austria.}
\email{ulisse.stefanelli@univie.ac.at}

\address[J. Zeman]{Faculty of Civil Engineering, Czech Technical University, Th\'{a}kurova 7, CZ--166~29 Praha~6, Czech Republic.}
\email{zemanj@cml.fsv.cvut.cz}

\date{\today}


\begin{abstract} 

 We investigate a variational theory for
magnetoelastic solids under the incompressibility constraint. The state of the system is
described by deformation and magnetization. While the former is
classically related to the reference configuration, magnetization
is defined in the deformed configuration instead. We
discuss the existence of energy minimizers without relying on higher-order
deformation gradient terms. Then, by
introducing a suitable positively
$1$-homogeneous dissipation, a quasistatic evolution model is
proposed and analyzed within the frame of energetic
solvability.

\end{abstract}

\maketitle

%\tableofcontents

\section{Introduction}

 Magnetoelasticity describes the mechanical behavior of solids under
magnetic effects. The magnetoelastic coupling is  caused by rotations of
small magnetic domains  from their original random orientation in
the absence of  a magnetic field. The orientation of these small
domains by the imposition of the magnetic field induces a deformation of the specimen. As
the intensity of the magnetic field is increased, more and more magnetic domains orientate
themselves so that their principal axes of anisotropy are collinear with the
magnetic field in each region and finally saturation is reached. We refer to
e.g.~\cite{brown,desim,desimone-james,jam-kin} for  a discussion on the 
foundations of magnetoelasticity.

 The mathematical modeling  of magnetoelasticity is a vibrant area
of research,  triggered by the interest on so-called {\it multifunctional} 
materials. Among these one has to mention rare-earth alloys
such as TerFeNOL and GalFeNOL as well as ferromagnetic shape-memory
alloys as Ni$_2$MnGa, NiMnInCo,
NiFeGaCo, FePt, FePd,
among others.  All these materials exhibit so-called {\it giant}
magnetostrictive behaviors as reversible strains as large as 10\% can be activated
by the imposition of relatively moderate magnetic fields. This strong
magnetoelastic coupling makes them relevant in a wealth of innovative applications including
sensors and actuators.  

Following the modeling approach of {\sc James \&
Kinderlehrer} \cite{jam-kin2}, the state of a magnetostrictive material is described by its deformation $y:\Omega \to \Rz^3$ from the reference
configuration $\Omega \subset \Rz^3$ and by its magnetization
$m:\O^y \to \Rz^3$ which is defined on the deformed configuration
$\O^y:=y(\Omega)$ instead. This discrepancy, often neglected by
restricting to small deformation regimes, is
particularly motivated here by the possible
large deformations that a magnetostrictive materials can experience.

 We shall here be concerned with the
total energy $E$ defined as
\begin{equation}
E(y,m)=\int_\Omega W(\nabla y, m\circ y) +\alpha
\int_{\O^y}|\nabla m|^2 + \frac{\mu_0}{2}\int_{\Rz^3}|\nabla
u_m|^2.\label{functional}
\end{equation}
Here, $W$ stands for the elastic energy density, the second term is
the so-called {\it exchange} energy and $\alpha$ is related to the
typical size of ferromagnetic texture. The last term represents
magnetostatic energy, $\mu_0$ is the permittivity of
void, and $u_m$ is the magnetostatic potential generated by $m$.
In particular, $u_m$ is a solution to the Maxwell equation
\begin{equation}
  \label{maxwell}
  \nabla \cdot (-\mu_0 \nabla u_m + \chi_{\O^y}m)=0 \ \ \text{in} \ \ \Rz^3,
\end{equation}
where $\chi_{\O^y}$ is the characteristic function of the deformed
configuration $\O^y$. We shall consider $E$ under
the a.e. constraints
\begin{equation}
  \label{constraints}
  \det \nabla y=1, \ \ |m|=1,
\end{equation}
which correspond to incompressibility and magnetic saturation (here
properly rescaled). Note that incompressibility is reputed to be a
plausible assumption in a vast majority of application \cite{desimone-james}.
 
The aim of this paper is twofold. At first, we concentrate on the static
problem. By assuming that $W$ is polyconvex and $p$-coercive in $\nabla y$ for
$p>3$ we check that $E$ admits a minimizer. This result is to be compared with
the discussion in {\sc Rybka \& Luskin} \cite{rybka-luskin} where weaker growth
assumptions on $W$ but a second-order deformation gradient is included. On the
contrary, no higher order gradient is here considered and we make full use of
the incompressibility constraint. In this direction, we shall mention also the
PhD thesis by {\sc Liakhova} \cite{Liakhova}, where the the dimension reduction
problem to thin films under the a-priori constraint $0<\alpha < \det
\nabla y <\beta$ is considered. This perspective has been numerically
investigated by {\sc Liakhova,  Luskin, \& Zhang} \cite{Luskin06,Luskin07}. More
recently, the incompressibility case has been addressed by a penalization method
from the slightly compressible case by {\sc Bielsky \& Gambin} \cite{Bielski10},
still by including a second-order deformation gradient term. We also  mention
the two-dimensional analysis by {\sc DeSimone \& Dolzmann} \cite{DeSimone} where
no gradients are considered and the existence of a zero energy state is checked
by means of convex integration techniques. Our discussion on the static problem
is reported in Section \ref{sec:energy}. Finally, let us point out that a closely related  static model on nematic elastomers  was recently analyzed  by {\sc Barchiesi}
\& {\sc DeSimone} in \cite{barchiesi-desimone}.  

A second focus of the paper is that of proposing a quasi-static evolution
extension of the static model. This is done by employing a
dissipation distance between magnetoelastic states which combines magnetic
changes with the actual deformation of the specimen. Note that the 
rate-independence  of this evolution seems well motivated for  fairly
wide range of frequencies of external magnetic fields.  We also ensure that the
elastic deformation is one-to-one at least inside the reference configuration
allowing for possible frictionless self-contact on the boundary. Let us mention
that  some models of rate-independent magnetostrictive
effects  were developed in \cite{bks,bs} in the framework 
magnetic shape-memory alloys  and  in
\cite{roubicek-kruzik,roubicek-kruzik1} for bulk ferromagnets. 
 
 We tackle the problem of ensuring the existence of quasi-static evolutions
under  frame of energetic solvability of rate-independent problems \`a la {\sc
  Mielke} \cite{mielke-theil-2,mielke-theil-levitas}. We restrict
ourselves to the isothermal situation. In particular  we assume  that
the process is sufficiently slow and/or the body thin in at least one direction
 so that the released heat can be  considered to be immediately
transferred to  the environment.  By relying on the classical
energetic-solution technology \cite{Mielke05} we prove that the implicit
incremental time discretization of the problem admits a time-continuous
quasi-static evolution limit. Details are given in Section
\ref{sec:dissipation}.

\section{ Energy} \label{sec:energy}

 Let the reference configuration $\O\subset\R^3$  be a bounded
Lipschitz domain.   Let us assume from the very beginning
$$p>3$$ 
and consider deformations
 $y\in W^{1,p}(\O;\R^3) \subset C(\overline \O;\Rz^3)$ 
where the bar denotes set closure. We impose
homogeneous boundary conditions by prescribing  that  $y=0$
 on $\Gamma_0\subset\partial\O$
 where $\Gamma_0$  has  a positive surface measure. Magnetization, representing  the   density
of magnetic spin moments,  is assumed to be defined on the open
set  $\O^y := y(\overline \O)\setminus y (\partial \overline \Omega)$
 and to have a fixed norm $1$ (note that our problem is
isothermal),  namely, $m:\O^y\to S^2$.  

% We define a total energy as 
% \begin{align}\label{functional}
% I(y,m):=\int_\O W(\nabla y(x),m(y(x))\,\md x- L(y)+\int_{\O^y}\alpha|\nabla m(z)|^2-h(z)\cdot m(z)\,\md z+\frac{\mu_0}{2}\int_{\R^3}|\nabla u_m(z)|^2\,\md z\ .\end{align}
% Here $W$ is stored energy density, i.e., the potential of the first Piola-Kirchhoff stress tenzor, $L$ is a linear bounded operator evaluating work of external loads on the specimen, $\alpha>0$ is an exchange energy constant,   $\mu_0$ is the vacuum permeability,  $h$ is the external magnetic field, and 
% $u_m$ is the magnetostatic potential generated by $m$. It satisfies the following equation:
% \be\label{maxwell}
% \nabla\cdot(-\mu_0\nabla u_m+\chi_{\O^y}m)=0 \mbox{ in $\R^3$\ ,}
% \ee
% where $\chi_{\O^y}$ is the characteristic function of $\O^y:=y(\O):=\{z\in\R^3;\ \exists x\in\O\mbox{ s.t. } z=y(x)\}$.

 The
incompressibility constraint  reads  $\det\nabla y=1$ almost
everywhere in $\O$.  In particular, this entails invertibility of
$y$ through  the Ciarlet-Ne\v{c}as condition \cite{ciarlet-necas} 
  which in our situation reads $|\O^y|=|\O|$.  Indeed, we have
  that 
$$ |\O^y| = \int_{\O^y}1 = \int_\Omega \det \nabla y = |\Omega|.$$

 We shall define the sets  
 \begin{align*}
y\in\mathbb{Y}&:=\{y\in W^{1,p}(\O;\R^3)\ |\ \det\nabla y=1\
\text{in}\ \Omega, y= 0
\mbox{ on }\Gamma_0,\ |\O^y|=|\O|\}\\
 m\in\mathbb{M}^y&:=\{m\in W^{1,2}(\O^y;\R^3);\ |m|=1\
\text{in}\ \Omega\}.
\end{align*}
 Note that, as $p>3$, the set $\mathbb{Y}$ is sequentially closed
with respect to the weak topology of $W^{1,p}(\O;\R^3)$. This indeed
 follows from the sequential continuity of the map
$y\mapsto\det\nabla y$ from $W^{1,p}(\O;\R^3) $ to $L^{p/3}(\O)$ (both
equipped with the weak convergence),  the weak closedness of the
Ciarlet-Ne\v{c}as condition \cite{ciarlet,ciarlet-necas}, and from the
compactness properties of the trace operator. 

 For the sake of brevity, we shall also define the set
$\mathbb{Q}$ as 
 $$\mathbb{Q}:=\{(y,m) \, | \,  (y,m)\in \mathbb{Y}\times \mathbb{M}^y\}\ .$$
  Moreover, we  say  that $\{(y_k,m_k)\}_{k\in\N}$ 
 $\mathbb{Q}$-converges  to $(y,m)\in\mathbb{Q}$ as $k\to\infty$  if the following three conditions hold
 \begin{subequations}\label{convergence}
 \begin{align}\label{convergence-y}
& y_k\rightharpoonup y \mbox{  in } W^{1,p}(\O;\R^3) ,\\
\label{convergence-m}
&\chi_{\O^{y_k}}m_k\to \chi_{\O^{y}}m \mbox{  in } L^2(\R^3;\R^3)  ,\\
\label{convergence-grad}
& \chi_{\O^{y_k}}\nabla m_k\rightharpoonup \chi_{\O^{y}}\nabla m
\mbox{  in } L^2(\R^3;\R^{3\times 3}) .
 \end{align}
 \end{subequations}
 Eventually, we say that a sequence 
$\{(y_k,m_k)\}_{k\in\N}\subset\mathbb{Q}$  is
$\mathbb{Q}$-bounded   if 
 $$\sup_{k\in\N}(\|y_k\|_{W^{1,p}(\O;\R^3)}+  \|\nabla
 m_k\|_{L^2(\O^{y_k};\R^{3\times 3})})< \infty . $$ 
 
  By following an argument from \cite[Lemma 3.5]{rybka-luskin},
 here simplified by the incompressibility assumption, we can show  that  $\mathbb{Q}$-bounded
  sequences  are $\mathbb{Q}$-sequentially-precompact. 
 
 \begin{proposition}\label{q-comp}
  Every
 $\mathbb{Q}$-bounded sequence admits a $\mathbb{Q}$-converging subsequence. 
 \end{proposition}  
 
\begin{proof}  Let $(y_k,m_k)$ be $\mathbb{Q}$-bounded. 
 The  compactness  in the $y$-component, i.e. \eqref{convergence-y},
 follows from  the weak closure of $\mathbb{Y}$. 
  
 Assume (without relabeling the subsequence)  that $y_k\rightharpoonup y
 $ in $W^{1,p}(\O;\R^3)$ and fix $\varepsilon>0$.  We denote by $\O^y$ the
 set  $\O^y_\varepsilon:=\{z\in\O^y;\, {\rm
   dist}(z,\partial\O)>\varepsilon\}$. As $p>3$ we have that
 $W^{1,p}(\O;\R^3)\hookrightarrow C(\bar\O;\R^3)$ compactly.  This in
 particular entails  that $\O^y_\varepsilon\subset\O^{y_k}$  for $k$
 sufficiently large. Hence, we infer that
 
 \begin{align*}
  \int_{\O^y_\varepsilon}|\nabla m_k| \le \int_{\O^{y_k}} |\nabla m_k|  <\infty\ .\end{align*}
 Taking into account that $|m_k| =1$ we get ( again  for a non-relabeled subsequence) that 
 $m_k\rightharpoonup m$ in $W^{1,2}(\O^y_\varepsilon;\R^3)$. Here the
 extracted  subsequence and  its limit  $m$ could depend on
 $\varepsilon$.  On the other hand,  as
 $\{\O^y_\varepsilon\}_{\varepsilon>0}$ exhausts $\O^y$ we  have  that 
 $m$  is defined  almost everywhere in $\O^y$.  By
 following the argument in 
 \cite[Lemma 3.5]{rybka-luskin} we  exploit the decomposition 
 \begin{align} 
   \|\chi_{\O^{y_k}}m_k-\chi_{\O^y}m\|_{L^2(\R^3;\R^3)}&\le\|(\chi_{\O^{y_k}}-\chi_{\O^y_\varepsilon})m_k\|_{L^2(\R^3;\R^3)}+\|\chi_{\O^y_\varepsilon}(m_k-m)\|_{L^2(\R^3;\R^3)}\nonumber\\
&+\|(\chi_{\O^y_\varepsilon}-\chi_{\O^y})m\|_{L^2(\R^3;\R^3)}\label{miserve}
.\end{align}
 We now check that  the above right-hand side goes to $0$ as $k\to\infty$ and
$\varepsilon\to 0$.  As to the first term,  since  $\overline{\O^y}$ is compact
  we have that   for any $\varepsilon>0$   there exists
  an open set $O_\varepsilon$ such that
 $O_\varepsilon\supset\overline{\O^y}$ and
 $|O_\varepsilon\setminus\O^y|<\varepsilon$. The uniform convergence
  $y_k\to y$  yields that $\O^{y_k}\subset O_\varepsilon$ for
 $k$ sufficiently large.  Therefore,
 $|O_\varepsilon\setminus\O^y_\varepsilon|$ can be made arbitrarily
 small if $\varepsilon$ is taken small enough, and the
 first term in  the right-hand side of  \eqref{miserve}
 converges to  $0$ as  $k\to\infty$ and
$\varepsilon\to 0$. The second term in the right-hand side of \eqref{miserve} goes to
$0$ with $k \to \infty$ as $m_k \to m$ strongly in $L^2(\O^y_\varepsilon;\R^3)$.
 As $|m|=1$ almost everywhere, the third term in the right-hand side of \eqref{miserve} is
bounded by $\| \chi_{\O^{y}} - \chi_{\O^y_\varepsilon}\|_{L^2(\R^3;\R^3)}$
which goes to $0$ as  
$\varepsilon\to 0$. 
 This shows  the convergence  \eqref{convergence-m}.

  A similar argument can then  be used to show that  
 \begin{equation*}
 \chi_{\O^{y_k}}\nabla m_k \rightharpoonup  \chi_{\O^y}\nabla m \mbox{ in } L^2(\R^3;\R^{3\times 3})\ ,
 \end{equation*}
 namely convergence  \eqref{convergence-grad}.
\end{proof}

 \begin{remark}\label{composition}\rm 
 Notice that the proof of the strong convergence of
 $\{\chi_{\O^{y_k}}m_k\}$  still holds  if we replace $\O$
 by  some  arbitrary measurable subset $\omega\subset\O$. 
Keeping in mind that $\det\nabla y_k=\det\nabla y=1$ almost everywhere
in $\O$,  for all $k\in\N$,  and that all mappings $y_k$ and $y$  are invertible,  we calculate
 $$
\int_\omega m_k \circ y_k=\int_{\R^3}\chi_{y_k(\omega)} m_k \to
\int_{\R^3}\chi_{y(\omega)} m =\int_\omega m \circ y
.$$ 
This shows  $m_k \circ y_k \rightharpoonup m \circ y$   in
$L^2(\O;\R^3)$.  As the $L^2$ norms converge as well, we get
strong convergence in $L^2(\O;\R^3)$.  Eventually, as $m_k$ takes
values in $S^2$ one has that  $m_k \circ y_k \rightharpoonup m \circ y$   in
$L^r(\O;\R^3)$ for all $r <\infty$ as well. 
\end{remark}

The following result is an  immediate  consequence of the linearity of 
the Maxwell equation  \eqref{maxwell}.
\bigskip
\begin{lemma}\label{maxwell-limit}
Let $\chi_{\O^{y_k}}m_k\to \chi_{\O^{y}}m$ in $L^2(\R^3;\R^3)$ and let
$u_{m_k}\in W^{1,2}(\R^3)$ be the solution of \eqref{maxwell}
corresponding to  $\chi_{\O^{y_k}}m_k$. Then $u_{m_k}\rightharpoonup
u_m$ in $W^{1,2}(\R^3)$ where $u_m$ is the solution of \eqref{maxwell}
 corresponding to  $\chi_{\O^{y}}m$.
\end{lemma}

 Let us finally enlist here our assumptions on the elastic
energy density $W$.

 \begin{subequations}\label{assumptions}
 \begin{align}
&  \exists c>0\, \forall F,m\,: -1/c+c|F|^p \le W(F,m),\label{growth}\\
&\forall R\in{\rm SO}(3)\,: W(RF,Rm)=W(F,m), \label{frame}\\
\label{even}
&\forall F,m\,: W(F,m)= W(F, \pm m),\\
&\label{polyconvexity}\forall F,m\,: W(F,m)=\haz W(F,{\rm cof}\, F, m),
\end{align}
 \end{subequations}
 where $\haz W:\R^{3\times 3}\times\R^{3\times 3}\times\R^3\to \R$ is  a continuous function  such that
 $\haz W(\cdot,\cdot,m)$ is convex for every $m\in  S^2$.  In
 particular, we assume material frame indifference \eqref{frame} and
 invariance under magnetic parity \eqref{even}.
 Recall that for $F \in \Rz^{3\times 3}$invertible one has  cof$\, F$ is defined as
 cof$\, F:=(\det F) F^{-\top}$. In the present incompressible case
 $\det F =1$ we
 simply have cof$\, F:=  F^{-\top}$. Eventually,  assumption \eqref{polyconvexity}   corresponds to the
 polyconvexity of the function  $W(\cdot,m)$
 \cite{ball77}.   Assumptions \eqref{assumptions} will be considered in
 all of the following, without explicit mention.

 \begin{theorem}[ Existence of minimizers]\label{existence-static}
 The energy $E$  
is lower semicontinuous  and coercive  with respect to  $\mathbb{Q}$-convergence.
 In particular, it attains  a  minimum on~$\mathbb{Q}$. 
\end{theorem}
 
 \begin{proof}
 Owing to the coercivity assumption \eqref{growth}, one
immediately gets that $E$ sublevels are 
$\mathbb{Q}$-bounded, hence $\mathbb{Q}$-sequentially compact due to
Proposition \ref{q-comp}. 
 
 The  magnetoelastic  term  in $E$ is weakly lower semicontinuous because of the
 assumptions \eqref{assumptions} on $W$, see  \cite{ball77,
   eisen}.
The  exchange energy term in $E$ is quadratic hence weakly lower
 semicontinuous. The magnetostatic term is  weakly lower
 semicontinuous by Lemma \ref{maxwell-limit}. 
The
 existence of a minimizer follows  from the direct
 method, e.g.~\cite{dacorogna}.
 
\end{proof}

 For the sake of notational simplicity in all of this section no
external forcing acting on the system was considered. It is however worth mentioning explicitly that the analysis extends
immediately to the case of the linear perturbation of the energy $E$
given by including the term
$$ - \Bigl(\int_{\O^y}\!\!\! h \cdot m + \int_\O f\cdot u +
\int_{\Gamma_t}g \cdot u\Bigr).$$
The first term is the so-called {\sc Zeeman} energy and $h\in
L^1(\O^y;\Rz^3)$ represents an external magnetic field. Moreover, $f\in L^q(\O;\Rz^3)$
is a body force, and $g \in L^q(\Gamma_t;\Rz^3)$ is a traction acting
on $\Gamma_t$ where $\Gamma_t\subset \partial \Omega$ is
relatively open, $\partial \Gamma_0 = \partial \Gamma_t$ (this last
two boundaries taken in $\partial \Omega$), and $1/p+1/q=1$.

Eventually, we could replace the homogeneous Dirichlet boundary
condition $y=0$ on $\Gamma_0$ with some suitable non-homogeneous
condition without difficulties.

\section{ Evolution}\label{sec:dissipation}
 Let us now turn to the analysis of quasi-static evolution driven by $E$. In
order to do so, one has to discuss dissipative effect as well. Indeed, under
 usual loading regimes , magnetically hard materials, 
experience  dissipation. On the other hand,
the  dissipation mechanism in ferromagnets can be influenced by impurities in the
material without affecting substantially the stored energy.   This allows us
to consider energy storage and dissipation as independent mechanisms. 

Our, to some extent simplified, standpoint is that the amount of
dissipated energy within the phase transformation from one pole to the other can be
described by a single, phenomenologically given number (of the
dimension J/m$^3$=Pa) depending on the coercive force $H_{\rm c}$ \cite{chikazumi}.
 Being interested in quasistatic, rate-independent processes 
we follow  \cite{mielke-theil,mielke-theil-2,mielke-theil-levitas}
 and  define the so-called dissipation distance between to
states $q_1:=(y_1,m_2)\in\mathbb{Q}$ and
$q_2:=(y_2,m_2)\in\mathbb{Q}$  by introducing   $\mathcal{D}:\mathbb{Q}\times\mathbb{Q}\to[0;+\infty)$ as follows
\begin{align*}
\mathcal{D}(q_1,q_2):=\int_\O H_{\rm c}|m_1(y_1(x))-m_2(y_2(x))|\,\md x .
\end{align*}
 Here, the rationale is that although the system
dissipates via magnetic reorientation only, elastic deformation also contributes
to dissipation as $m$ lives in the deformed configuration. 
 
Assume, for simplicity, that the evolution of the specimen during a
process time interval $[0,T]$ is driven by  the  time-dependent 
loadings 
 \begin{align*}%\label{loads}
f &\in C^1([0,T];L^q (\O;\R^3)), \\ g &\in C^1([0,T];L^q (\Gamma_t;\R^3)), \\ 
h&\in C^1([0,T];L^1(\Rz^3;\R^3)), 
\end{align*}
 so that we can write a (time-dependent) energy functional 
 $\mathcal{E}:[0,T]\times \mathbb{Q} \to (-\infty,\infty)$ 
as 
 \begin{align}\label{functional-timedep}
\mathcal{E}(t,q):= E(q) -
\Bigl(
\int_{\O^y}\!\!\! h(t) \cdot m + \int_\O f(t)\cdot u + \int_{\Gamma_t}g(t)
\cdot u
\Bigr)
.\end{align}

Our aim is to find  an energetic solution  corresponding to the
energy and dissipation functionals $(\mathcal{E},\mathcal{D})$ \cite{mielke-theil-2,mielke-theil-levitas}, that
is an everywhere defined  mapping $q:[0,T]\to\mathbb{Q}$
such that
\begin{subequations}\label{solution}
\begin{align}
&\forall\, t\in[0,T], \ \forall\, \tilde q\in\mathbb{Q}:\  \mathcal{ E }(t,q(t))\le \mathcal{ E }(t,\tilde q)+\Dcal(q(t),\tilde q) ,\label{stability}\\ 
&\forall\, t\in[0,T]:\,\mathcal{ E }(t,q(t))+{\rm Var}(\Dcal,q;0,t)=
\mathcal{ E }(0,q(0))+\int_0^t\partial_t\mathcal{ E
}(\theta,q(\theta))\,{\rm d}\theta,\label{e-equality}
\end{align}
\end{subequations}
 where we have used the notation    \begin{align*}%\label{Var-R}
{\rm Var}(\Dcal,q;s,t):=\sup\sum_{i=1}^J\Dcal(q(t_{i-1}),q(t_i))
\end{align*}
the supremum  being taken  over all partitions of
$[s,t]$ in the form $\{s=t_0<t_1<...<t_{J-1}<t_J=t\}$.  Condition
\eqref{stability} is usually referred to as the (global) stability of
state $q$ at time $t$. For the sake of convenience we shall call {\it
  stable} (at time $t$) a state fulfilling \eqref{stability} and 
denote by $\mathbb{S}(t) \subset \mathbb{Q}$  the set of stable states.  The scalar
relation \eqref{e-equality} expresses the conservation of energy
instead. We shall now state the existence result.

\begin{theorem}[Existence of energetic solutions]
Let   $q_0 \in  \mathbb{S}(0)$.  Then, there  exist 
an energetic solution  corresponding  to $(\mathcal{ E
  },\mathcal{D})$, namely  a trajectory   
 $q:=(y,m):[0,T]\to\mathbb{Q}$ such that  $q(0)=q_0$  and
 \eqref{solution}  are  satisfied. Additionally,  $q$ is uniformly
 bounded in $\mathbb{Q}$ and   $m\circ y\in BV(0,T;L^1(\O;\R^3))$.
\end{theorem}

\begin{proof}[Sketch of the proof]
 This argument follows the by now classical argument for existence of
energetic solutions. As such, we record here some comment referring
for instance to
 \cite{francfort-mielke,Mielke05} for  the details.  Starting from
 the stable  initial condition $ q_0  \in \mathbb{S}(0)$ we
(semi)discretize the problem in time  by means of  
a partition $0=t_0<t_1<\ldots <t_N=T$ of $[0,T]$  such that the diameter  $\max_{i}(t_{i}-t_{i-1})\to 0$ as $N\to\infty$.   This gives us a sequence ${q_k^N}$  such that $q_0^N:=q_0$ and $q_k^N$, $1\le k\le N$, is a solution to the  following  minimization problem for $q\in\mathbb{Q}$
\begin{align}
  \mbox{ minimize } \mathcal{ E
    }(t_k,q)+\mathcal{D}(q,q^N_{k-1}). \label{incremental-min}
\end{align}
 The existence of a solution to \eqref{incremental-min} follows
form Theorem \ref{existence-static} combined with the lower
semicontinuity of $\mathcal{D}$. In particular,  Remark~\ref{composition}.
implies that the dissipation term in \eqref{incremental-min} is
continuous with respect to the weak convergence in $\mathbb{Q}$.
 We now record that minimality and the triangle inequality entail that the  obtained solutions are stable, i.e.,
$q^N_k\in\mathbb{S}(t_k)$ for all $k=0,\ldots,N$.  Let us define 
the right-continuous piecewise interpolant 
$q^N:[0,T]\to\mathbb{Q}$ as
\begin{align*}
  q^N(t):=\begin{cases}
    q^N_k &\mbox{ if $t\in  [t_{k-1},t_{k})$ if $k=1,\ldots,
      N$}, \\
    q^N_N &\mbox{ if $t=T$.}
  \end{cases}
\end{align*} 
Following \cite{Mielke05} we  can establish  for all $N\in\N$ the 
a-priori estimates
\begin{subequations}
  \begin{align}
    \|y^N\|_{L^\infty(0,T);W^{1,p}(\O;\R^3)}\le C ,\\
    \|\chi_{\O^{y^N}}\nabla m^N\|_{L^\infty((0,T);L^2(\R^3;\R^3))}\le C,\\
    \|\chi_{\O^{y^N}} m^N\|_{L^\infty((0,T);L^\infty(\R^3;\R^3))}\le C ,\\
    \|m^N\circ y^N\|_{BV(0,T; L^1(\O;\R^3))}\le C.
  \end{align}
\end{subequations}
 These 
a-priori estimates  together with a suitably  generalized 
version of 
Helly's selection principle \cite[Cor.~2.8]{mielke-theil-levitas} 
entail that, for some not relabeled subsequence, we have
$q^N \to q$ pointwise in $[0,T]$ with respect to the weak topology of
$\mathbb Q$. This convergence suffices in order to prove
that indeed the limit trajectory is stable, namely $q(t)\in \mathbb{Q}(t)$ for all $t\in
[0,T]$. Indeed, this follows from the lower semicontinuity of
$\mathcal{E}$ and the continuity of $\mathcal{D}$. 

Moreover, by exploiting minimality we readily get that 
\begin{align*}
 % \int_{t_{k-1}}^{t_k}\partial_t\mathcal{ E  }(\theta,q_{k}^N)\,{\rm d}\theta\le
 \mathcal{ E }(t_k,q^N_k)+\mathcal{D}(q^N_k,q^N_{k-1})-\mathcal{ E }(t_{k-1},q^N_{k-1})\le 
  \int_{t_{k-1}}^{t_k}\partial_t\mathcal{ E }(\theta,q_{k-1}^N)\,{\rm d}\theta\ .
\end{align*}
 Taking the sum of the latter on $k$ we readily check that the
one-sided inequality in relation \eqref{e-equality} holds for $t=T$.
The converse energy inequality (and hence \eqref{e-equality} for all $t\in[0,T]$) follows from the stability $q(t)\in
\mathbb{S}(t)$ of the limit trajectory  by \cite[Prop.
5.6]{Mielke05}.

 Note that   the previous existence result can be adapted to
the case of time-dependent  non-homogeneous  Dirichlet
boundary conditions  by following the corresponding argument
developed in \cite{francfort-mielke}. 
\end{proof}

%%%%%%%%%%%%%%%%%%%%%%%%%%%%%%%%%%%%%%%%%%%%%%%%%%%%%%%%%%%%%%%%%%%%%%%%%%%%%%%%%%%%%%%%%%%%%%%%%%%%%%%%%%%%
 
\section*{Acknowledgment.} This work was initiated during a visit of MK and JZ
in the  IMATI CNR Pavia. The hospitality of the institute is gratefully
acknowledged. MK and JZ acknowledge the support by GA\v{C}R through the projects
P201/10/0357, P105/11/0411, and 13-18652S.

\end{document}